\documentclass[12pt,a4paper,oneside,reqno,notitlepage]{amsart}

\usepackage{amsfonts}
\usepackage{amscd,amsmath}
\usepackage{amssymb}
\usepackage{setspace}
\newtheorem{thm}{Theorem}[section]
\newtheorem{lem}[thm]{Lemma}

\newtheorem{prop}[thm]{Proposition}
\newtheorem{definition}[thm]{Definition}

\theoremstyle{remark}
\newtheorem{rem}{Remark}[section]
\newtheorem*{rem*}{Remark}

\newcommand{\comment}[1]{}
\begin{document}

\title{Existence of a Unique group of finite order}
\author{ Sumit Kumar Upadhyay$^1$ and Shiv Datt Kumar$^2$\vspace{.4cm}\\
DEPARTMENT OF MATHEMATICS\\
MOTILAL NEHRU NATIONAL INSTITUTE OF TECHNOLOGY \\
ALLAHABAD, U. P., INDIA}
\thanks{ $^1$upadhyaysumit365@gmail.com, $^2$sdt@mnnit.ac.in}

\begin{abstract}
Let $n$ be a positive integer. Then cyclic group $Z_n$ of order $n$ is the only group of order $n$  iff g.c.d. $(n,\varphi(n))=1$, where $\varphi$ denotes the Euler-phi function. In this article we have given another proof of this result using the knowledge of semi direct product and induction.
\end{abstract}

\maketitle

\section{Introduction}
One of the main problem in group theory is to classify groups upto isomorphisms.  For example one may ask: How many non-isomorphic groups of order $n$ are there? and what are they ?.
The Fundamental theorem of Abelian groups  says that every finite Abelian group  is the direct product of cyclic  groups of prime power order, which  classifies the non-isomorphic Abelian groups of order $n$.  But there is no classification of non-isomorphic non-Abelian groups of order $n$.
For every natural $n$, there must be a cyclic group $\mathbb{Z}_n$ of order $n$. If $n$ is prime, then $\mathbb{Z}_n$ is the only group of order $n$.  Also if $n = pq$, where $p$ \& $q$ are distinct primes such that $p$ does not divides $(q-1)$, then by the Sylow theorem, $\mathbb{Z}_n$ is the only group of order $n$.  Thus there is a natural question: what are conditions for existence of  a unique group of order $n$.
Dieter Jungnickel proved (\cite{DJ}) that if greatest common divisor (g.c.d.) $(n,\varphi(n))=  1$, then  $Z_n$ is the only group of order $n$.
In this article we have given another proof using the knowledge of the semi direct product and induction.

\begin{definition}
  Let $H$ and $K$ be two groups and let $\phi$ be a homomorphism from $K$ to $Aut(H)$.
            Let $G=\{(h,k)\vert  h \in H \& k \in K \}$.
           Define multiplication on $G$ as
                  $(h_1,k_1) (h_2,k_2)=(h_1\phi(k_1)(h_2),k_1k_2)$.
     This multiplication makes $G$  a group of order $\mid G \mid = \mid H \mid \mid K \mid$, where $(1,1) $ is the identity of $G$ and $ (h,k)^{-1}=(\phi(k^{-1})(h^{-1}),k^{-1}) $ is the inverse of $(h,k)$.
            The group $G$ is called the semi-direct product of $H$ and $K$ with respect to $\phi$.
\end{definition}
\begin{rem}
   If $\phi$ is a trivial homomorphism from $K$ to $Aut(H)$, then $\phi(k_1) = I$ (identity),  $\forall k_1 \in K$. Therefore $(h_1,k_1) (h_2,k_2)=(h_1\phi(k_1)(h_2),k_1k_2) = (h_1 h_2, k_1k_2)$. Hence  semi direct product of $H$ and $K$ is the direct product of $H$ and $K$.
\end{rem}

\begin{lem}
For a given natural number $n$, if g.c.d. $(n,\varphi(n))= 1$, then $n$ must be square free, where $\varphi(n)$ is the Euler phi function.
\end{lem}

\begin{proof}
Suppose contrary that  $n = p^{\alpha}m$, where $\alpha>1$ and $p^{\alpha}$  does not divide $ m $.           Then $\varphi(n)=(p^{\alpha}-p^{\alpha-1})\varphi(m)$.
This shows that $p^{\alpha-1}\mid \varphi(n)$ and $p^{\alpha-1}\mid n$.
     Hence $(n,\varphi(n)) = p^{\alpha-1}$, which is a contradiction.
 Hence $n$ is  square free.
\end{proof}

\begin{prop}
If cyclic group  $Z_n$ is the only group of order $n$, then $n$ must be square free.
\end{prop}

\begin{proof}
By the Fundamental theorem of Abelian groups, we know that if $n = p_1^{n_1}\ldots p_r^{n_r}$, where $ p_1, p_2,\ldots, p_r $ are distinct primes,
the number of non-isomorphic Abelian groups of order $n$ is $p(n_1)\ldots p(n_r)$, where $p(n_i)$ denotes number of partitions of $n_i$. Since $Z_n$ is unique group of order $n$, therefore $n$ must be square free.
\end{proof}

\begin{rem}
Let $G$ be a group of order $n$, and $p$ be a prime dividing $n$ and $P$ be a Sylow $p$-subgroup. If $\mid P \mid = p$, then every non identity element of $P$
has order $p$ and every element of $G$ of order $p$ lies in some conjugate of $P$. By the Lagrange's theorem distinct conjugates of $P$ intersect in the identity,
hence in this case the number of elements of $G$ of order $p$ is $n_p(p-1)$, where $n_p$ denotes number of distinct conjugates of $P$.

        Suppose Sylow $p$-subgroups for different primes $p$ have prime order and we assume none of these are normal. Then
the number of elements of prime order is greater than $ \mid G \mid $. This contradiction  would show that at least
 one of the $n_p's$ must be $1$ (i.e., some Sylow subgroup is normal in $G$).

     For example, suppose $\mid G \mid = 1365 =3.5.7.13$. If $G$ were simple, we must have $n_3 = 7, n_5 = 21, n_7 = 15$ and $n_{13} = 104$. Thus
\begin{center}
the number of elements of order $3$ =  $7.2    =  14$

the number of elements of order $5$ = $21.4   =  84$

the number of elements of order $7$ = $15.6   =  90$

the number of elements of order $13$ = $105.12  = 1260$

  Thus the number of elements of prime order = $1448  >  \mid G \mid$.

\end{center}

\end{rem}
\begin{prop}
Let $G$ be a group of order $p_1p_2\ldots p_s$, where $p_1,p_2,\ldots, p_s$ are distinct primes. Then $G$ is solvable.
\end{prop}
\begin{proof}
Proof is by induction on $s$.
If $s=1$, then $G$ is a cyclic group. Then $G$ must be solvable.Thus result is true for $s=1$.
By induction hypothesis assume that result is true for $s = k$ i.e. every group of order $p_1p_2\ldots p_k$ is solvable. Now to prove the result for
$s = k+1$. Since every Sylow $p$-subgroups for different primes $p$ have prime order, so by the Remark 1.2, some Sylow subgroup is normal in $G$.
 Suppose $H$ is a normal Sylow $p_i$-subgroup of $G$ for some $i$.
Then $G/H$ is a group of order $p_1p_2 \ldots p_{i-1}p_{i+1} \ldots p_{k+1}$. By induction assumption $G/H$ is solvable group.
 Since $H$ is a cyclic subgroup of $G$, $H$ is solvable.
   Since $H$ and $G/H$ are solvable, therefore $G$ is also solvable.

Hence by induction hypothesis result is true for every  $s$.

\end{proof}

\begin{thm}
Let $n$ be a positive integer. Then cyclic group $Z_n$ of order $n$ is the only group of order $n$  iff g.c.d. $(n,\varphi(n))=1$, where $\varphi$ denotes the Euler-phi function.
\end{thm}

\begin{proof}
First assume that cyclic group $Z_n$ of order $n$, is the only group of order $n$.
Then by the Proposition 1.3, $n = p_1p_2\ldots p_r$ where $p_1, p_2,\ldots, p_r$ are distinct primes.
So $\varphi(n)=(p_1-1)(p_2-2)\ldots(p_r-1)$.
Suppose g.c.d. $(n,\varphi(n))\neq 1$. Then there must exist two primes $p_i$ and $p_j$ such that $p_i\mid p_j-1$, where $i\neq j$ for some $i, j  \in \{1, 2, \ldots,  r\}$.
Then $n = p_i p_j m$(say). Hence there exists a non-Abelian group $H$ of order $p_i p_j$.
Take another group $K = Z_m$. Then the semi-direct product of $H$ and $K$ is a non-Abelian group of order $n$, which contradicts our assumption that $Z_n$ is the only group
 of order $n$, so our supposition is wrong.
Therefore g.c.d. $(n,\varphi(n)) = 1$.

Conversely,  suppose  g.c.d. $(n,\varphi(n)) = 1$. Then by the Lemma 1.2, $n$ must be square free i.e. $n = p_1 p_2 \ldots p_r$ where $p_1,\ldots, p_r$ are distinct primes.
  We show that $Z_n $ is the only group of order $n$. Proof is by  induction on $r$.
If $r = 1$. Then $n = p_1$ i.e. $n$ is a prime. Since every  group of prime order is cyclic, $Z_n$ is the only group of order $n$. Thus result is true for $r = 1$.
Suppose by induction hypothesis that the result is true for $r = k$ i.e. for $n = p_1 p_2 \ldots p_k$, $Z_n $ is the only group of order $n$. Now to prove the result for
 $r = k+1$ i.e.  for $n = p_1 p_2\ldots p_{k+1}$,
Take $H = Z_m$,  where $m = p_1 p_2 \ldots p_k$ such that g.c.d. $(m,\varphi(m))= 1$  and $K = Z_{p_{k+1}}$. Consider the semi direct product of $H$ and $K$.
It exists, because semi direct product
 of any two groups exist.
Since $H$ is a cyclic group, so $o(Aut(H))=(p_1-1)\ldots (p_k-1)$.
Then any homomorphism from $K $ to $Aut(H)$ must be a trivial homomorphism, because if not, then g.c.d. $((p_1-1)(p_2-1) \ldots (p_k-1), p_{k+1})\neq 1$.
So g.c.d. $((p_1-1)(p_2-1) \ldots (p_k-1), p_{k+1})$ must be equal to $p_{k+1}$. Therefore $p_{k+1}\mid(p_1-1)(p_2-1)\ldots (p_k-1)$, so $p_{k+1}\mid(p_1-1)(p_2-1)\ldots (p_{k+1}-1)$
 i.e. $p_{k+1} $ divides $\varphi (n)$, which is not possible because g.c.d. $(n,\varphi(n)) = 1$. Thus homomorphism must be  trivial.
Therefore by the Remark 1.1, semi-direct product of $H$ and $K$ is the direct product of $H$  and $K$, which is a cyclic group of order $p_1 p_2\ldots p_{k+1}$,
for $p_1, p_2, \ldots, p_{k+1}$ are all distinct primes.
Since $H$ is the only group of order $m$,  $Z_{p_1 p_2 \ldots p_{k+1}}$ is the only group of order $n = p_1 p_2 \ldots p_{k+1}$.
Hence by induction hypothesis result is true for every $r$.

Now  let $G$ be any group of order $p_1 p_2 \ldots p_k p_{k+1}$ where $p_1, p_2, \ldots, p_k, p_{k+1}$ are all distinct primes. Then we want to prove the
existence of a normal subgroup $H$ of order $p_1p_2 \ldots p_k$ and a subgroup $K$ of order $p_{k+1}$. Since all $p_{i}$'s are primes, by Sylow theorem $G$ has
 a subgroups of order $p_{i}$, $\forall i$. Now we show the existence of $H$.

  By the Proposition 1.4,  $G$  is solvable, so  commutator $G'$ is either identity subgroup or a proper subgroup of $G$ (because if $G'= G$,
 then $G$ can not be solvable). If $G' = \{e\}$, where $e$ is the identity of $G$. Then $G$ must be an Abelian group.
 Therefore $G$ is  cyclic by the Fundamental theorem of Abelian groups. Suppose $G'$ is a proper subgroup of $G$. Then $o(G') = p_1p_2 \ldots p_i$,
where $1 \leq i < k+1$. So the factor group $G/{G'}$ is an Abelian group of order $p_{i+1}\ldots, p_{k+1}$. Therefore by Cauchy's theorem
 $G/{G'}$ has a normal subgroup $H/{G'}$ of order $p_{i+1}\ldots, p_{k}$. Hence $H$ is also a normal subgroup of $G$ and $o(H) = p_1 p_2 \ldots p_{k}$.
 This proves the existence of a normal subgroup $H$.  Now consider the semi direct product of $H$ and $K$, which is isomorphic to $G$ for $o(G) = o(H K)$.
 Hence the result is proved.
\end{proof}

\noindent \textbf{Acknowledgment} We sincerely thank Professor Alain Valette for his remarks and comments on the proof of the Theorem 1.5.


\begin{thebibliography}{XXX}

\bibitem[1]{DJ}
Dieter Jungnickel,
\newblock {On the uniqueness of the cyclic group of order $n$},
\newblock {\em American Mathematical Monthly}, {\bf 99 (6)} (1992), 545 - 547.
\bibitem[2]{RJL}
Ramji Lal,
\newblock { Algebra Vol I},
\newblock Shail Publication, (2002).

\bibitem[3]{DF}
Dummit and  Foote,
\newblock {Abstract Algebra},
 (1995).


\end{thebibliography}
\end{document}